\documentclass[twoside,a4paper]{amsart}
\usepackage{style}
\usepackage{mathabx}
\usepackage{tikz}
\usepackage{stmaryrd}
\usepackage{graphicx}

\usepackage[left=3cm, right=3cm]{geometry}

\usepackage{xcolor}
\usepackage{pdfsync}
\usepackage{xcolor}
\definecolor{webgreen}{rgb}{0,.5,0}
\definecolor{webbrown}{rgb}{.6,0,0}
\definecolor{RoyalBlue}{cmyk}{1, 0.50, 0, 0}
\usepackage[colorlinks=true, breaklinks=true, urlcolor=webbrown, linkcolor=RoyalBlue, citecolor=webgreen,backref=page]{hyperref}

\usepackage{courier} 
\normalfont
\usepackage{epsfig, graphicx, subfigure}
\usepackage{amsmath, amssymb}

\def\ge{\geqslant}
\def\le{\leqslant}

\newtheorem{theorem}{Theorem}[section]

\newtheorem{corollary}[theorem]{Corollary}
\newtheorem{lemma}[theorem]{Lemma}

\theoremstyle{remark}

\definecolor{darkbrown}{RGB}{150,1,33} 
\numberwithin{equation}{section}
\usepackage{hyperref}
\begin{document}

\title[Convergence  \ldots \ldots]{
Convergence of sparse square-summable NLFT}

\begin{abstract}  We study convergence of SU(1,1) and SU(2) NLFT with sparse $\ell^2(\Z)$ data. The asymptotics of the associated  polynomials orthogonal on the unit circle is obtained as a corollary. 
\end{abstract} \vspace{1cm}

\author[Sergey A. Denisov]{Sergey A. Denisov}
\address{Department of Mathematics, University of Wisconsin-Madison, 480 Lincoln Dr., Madison, WI 53706, USA}
\email{\href{mailto:denissov@wisc.edu}{denissov@wisc.edu}}

\thanks{
This research was supported by   NSF grant DMS-2450716, the Simons Fellowship in Mathematics, the Simons Travel Support for Mathematicians Award, and the Van Vleck Professorship Research Award. 
}



\maketitle

\setcounter{tocdepth}{3}


\section{ The $SU(1,1)$ and $SU(2)$ NLFT and main results.}

\medskip

This note is a companion to \cite{den26} where the $SU(1,1)$ nonlinear Fourier transform (NLFT) with general square-summable coefficients was studied and  we will  use some results and notation from that work. Denote $\N=\{1,2\ldots\}, \Z^+=\{0,1,2,\ldots\}$ and $\Z$ will stand for the set of all integers.
  Let $F=\{F_n\}|_{n\in \Z}$ be a double-infinite sequence of complex numbers and $|F_n|<1,\forall n\in \Z$. Taking $N\in \Z^+$, we define a compactly supported $F^{\langle N\rangle}$ by truncation: $F^{\langle N\rangle}_n=F_n\cdot \chi_{|n|\le N}, n\in \Z$, where $\chi_E$ denotes the characteristic function of a set $E$. Consider $2\times 2$ matrices $\widetilde X_n(z,F^{\langle N\rangle}), n\in \Z$ defined by the recursion
\begin{equation}\label{rec5}
\widetilde X_n= \frac{1}{(1-|F_n^{\langle N\rangle}|^2)^{\frac 12}}
\left(\begin{array}{cc}
1& \overline{F_n^{\langle N\rangle}}z^{-n}\\
F_n^{\langle N\rangle} z^n &1
\end{array}\right)\widetilde X_{n-1}, \quad z\in \C
\end{equation}
with $\widetilde X_n=I:=\left(\begin{smallmatrix}1&0\\0&1
\end{smallmatrix}
\right)$ for $n< -N$. One has $\widetilde X_n\in SU(1,1)$ for $z\in \T$ and a more careful study of this  recursion  gives (see p.5 in \cite{tt}) 
\[
\widetilde X_n(z,F^{\langle N\rangle})=:\left(\begin{array}{cc}
\frak{a}_n(z,F^{\langle N\rangle})& \frak{b}^{(*)}_n(z,F^{\langle N\rangle})\\
\frak{b}_n(z,F^{\langle N\rangle}) &\frak{a}^{(*)}_n(z,F^{\langle N\rangle})
\end{array}\right)\,,
\]
where \begin{equation}\label{oper1}
f^{(*)}(z):=\overline{f(\bar{z}^{-1})}.
\end{equation} 
Clearly, $f^{(*)}(z)=\overline{f(z)}$ if $z\in \T$. We also have
\begin{equation}\label{susu}
|\frak{a}_n|^2=|\frak{b}_n|^2+1, \, z\in \T\,.
\end{equation}
We put forward a definition: \smallskip

\noindent {\bf Definition.} {\it For $F^{\langle N\rangle}$, define the $SU(1,1)$ nonlinear Fourier transform (NLFT) on $\Z$ as the map 
\[
F^{\langle N\rangle}\mapsto \overbrace{F_{su(1,1)}^{\langle N\rangle}}:=\left(\begin{array}{c}
\frak{a}(z,F^{\langle N\rangle})\\
\frak{b}(z,F^{\langle N\rangle}) 
\end{array}\right)\,,
\]
where $\frak{a}(z,F^{\langle N\rangle}):=\frak{a}_\infty(z,F^{\langle N\rangle})$ and 
$\frak{b}(z,F^{\langle N\rangle})=\frak{b}_\infty(z,F^{\langle N\rangle})$.}
\smallskip

 In fact, since $F_n^{\langle N\rangle}=0$ for $n>N$,  we get 
\[
\left(\begin{array}{c}
\frak{a}(z,F^{\langle N\rangle})\\
\frak{b}(z,F^{\langle N\rangle}) 
\end{array}\right)=\left(\begin{array}{c}
\frak{a}_n(z,F^{\langle N\rangle})\\
\frak{b}_n(z,F^{\langle N\rangle}) 
\end{array}\right), \quad \forall n\ge N\,.
\]
We refer the reader to  \cite{tt,den26,kov,kov1,st,musc,Jelena} for additional references and discussion about the NLFT. 
Our first result is the following theorem.
\begin{theorem}\label{t1}Suppose $\{n_j\}|_{j=0}^\infty$ is a subsequence in $\N$ that satisfies 
\begin{equation}\label{spar}
\frac{n_{j+1}}{n_j}\ge q, \quad \forall j\in \Z^+, \quad n_0=1\,,
\end{equation}
where $q\ge 2$. Take $\{\delta_j\}\in \ell^2(\Z^+)$ such that $\|\left\{\delta_j\right\}\|_{\ell^\infty(\Z^+)}<1$. Let $F_{n_j}=\delta_j, \, \forall j\in \Z^+$ and set $F_n=0$ for all other $n\in \Z$. For any $p\in [2,\infty)$, the limit
 \[
 \lim_{N\to\infty}\overbrace{F_{su(1,1)}^{\langle N\rangle}}
 \]
 exists in $L^p(\T,m)$.

\end{theorem}

\noindent {\bf Remark.}  Clearly, in our theorem, the sparse sequence $F$ can be taken infinite in both directions, e.g., $\{n_j\}=\{\pm 2^s\}, s\in \Z^+$ vs. the standard $n_j=2^j, j\in \Z^+$.
The pointwise convergence of $SU(1,1)$ NLFT with sparse $\ell^2$ data has been studied in \cite{Jelena}. In particular,  the a.e. convergence was established there. Later, we will comment on how that can be obtained from our argument.\medskip

\noindent {\bf Remark.} Taking $\delta_j>0$, $\{\delta_j\}\notin \ell^1(\Z^+)$ and $z=1$ in Theorem \ref{t1} shows that the $L^\infty(\T,m)$ norm of NLFT can grow infinitely.
\medskip

We make a few observations before starting with the proof of Theorem \ref{t1}. Without loss of generality, we can assume that $\|\{\delta_j\}\|_{\ell^2}\le \frac 12$. Consider $Y_n(z,F)$ that solves the recursion
\begin{equation}\label{rec8}
Y_n= 
\left(\begin{array}{cc}
1& \overline{F_n}z^{-n}\\
F_n z^n &1
\end{array}\right) Y_{n-1}, 
\end{equation}
with $Y_0=I$. Clearly, 
$\widetilde X_n=\prod_{j=0}^n(1-|F_j|^2)^{-\frac 12}Y_n$ for $n\le N$ and, therefore, it is enough to show that $\lim_{n\to\infty} Y_n(z,F)$ converges for a.e. $z\in \T$. We denote
\[
Y_n(z,F)=:\left(\begin{array}{cc}
\widetilde{\frak{a}}_n(z,F)& \widetilde{\frak{b}}^{(*)}_n(z,F)\\
\widetilde{\frak{b}}_n(z,F) &\widetilde{\frak{a}}^{(*)}_n(z,F)
\end{array}\right)\,.
\]
Notice that $Y_{n}=Y_{n-1}$ if $F_n=0$. We will show that $\lim_{n\to\infty}\widetilde{\frak{b}}_n$ and $\lim_{n\to\infty}\widetilde{\frak{a}}_n^{(*)}$ converge for a.e. $z\in \T$ but we need some notation first.
In what follows, $m$ will denote the normalized Lebesgue measure on the unit circle $\T$, i.e., $dm=(2\pi)^{-1}d\theta, \theta\in [0,2\pi)$. Also, for the function $f\in L^1(\T,m)$, we write $\widehat f_n=\int_\T fe^{- in\theta}dm, n\in \Z$ for the Fourier transform. Similarly, given a sequence $\{G_n\}$, we will write $\widecheck G=\sum_{n\in \Z} G_ne^{in\theta}$ for the inverse Fourier transform whenever the sum makes sense. 

\begin{lemma}Assume that $\|\{\delta_j\}\|_{\ell^2}\le \frac 12$, then
\[
\sup_n\|\widetilde{\frak{a}}_n^{(*)}(z,F)\|_{L^2(\T,m)}^2=1+O(\|\{\delta_j\}\|_{\ell^2}^2),\quad \sup_n\|\widetilde{\frak{b}}_n(z,F)\|_{L^2(\T,m)}^2=O(\|\{\delta_j\}\|_{\ell^2}^2)\,.
\]
\end{lemma}
\begin{proof}
Denote for shorthand $\alpha_j:=\widetilde{\frak{a}}^{(*)}_{n_j},\beta_j:=\widetilde{\frak{b}}_{n_j}$ for $j\in \Z^+$ and $\alpha_{-1}:=1,\beta_{-1}:=0$. Then, for $z\in \T$, we have
\begin{equation}\label{sd1}
\alpha_{j}=\alpha_{j-1}+{\delta}_j z^{n_j}\overline{\beta}_{j-1}, \quad \beta_{j}=\beta_{j-1}+\delta_j z^{n_j}\overline{\alpha}_{j-1}\,.
\end{equation}
Due to $\supp \widehat\alpha_{j-1}\subset [0,n_{j-1}-1]$ and $\supp \widehat\beta_{j-1}\subset [1,n_{j-1}]$ (see \cite{tt}, p.6, Lemma 2), we get the $L^2(\T,m)$-orthogonality of the terms due to \eqref{spar}. Then, for $j\in \Z^+$, we get
\begin{eqnarray}
\|\alpha_j\|_{L^2(\T,m)}^2=\|\alpha_{j-1}\|_{L^2(\T,m)}^2+|\delta_j|^2\|\beta_{j-1}\|_{L^2(\T,m)}^2, \quad \|\alpha_{-1}\|_{L^2(\T,m)}=1,\\
\|\beta_j\|_{L^2(\T,m)}^2=\|\beta_{j-1}\|_{L^2(\T,m)}^2+|\delta_j|^2\|\alpha_{j-1}\|_{L^2(\T,m)}^2, \quad \|\beta_{-1}\|_{L^2(\T,m)}=0\,.
\end{eqnarray}
Since $\{\delta_j\}\in \ell^2(\Z^+)$, we can add these two identities and iterate to get 
\[
\|\alpha_j\|_{L^2(\T,m)}^2+\|\beta_j\|_{L^2(\T,m)}^2=\prod_{s\le j}(1+|\delta_s|^2)\,.
\]
Hence, for $j\in \N$, we have
\begin{eqnarray}
\|\alpha_j\|_{L^2(\T,m)}^2= \|\alpha_{j-1}\|_{L^2(\T,m)}^2+O\left(|\delta_j|^2\prod_{s\le j-1}(1+|\delta_s|^2)\right), \quad \|\alpha_{-1}\|_{L^2(\T,m)}=1,\\
\|\beta_j\|_{L^2(\T,m)}^2= \|\beta_{j-1}\|_{L^2(\T,m)}^2+O\left(|\delta_j|^2\prod_{s\le j-1}(1+|\delta_s|^2)\right), \quad \|\beta_{-1}\|_{L^2(\T,m)}=0\,.
\end{eqnarray}
Iterating these formulas, we get  our result.
\end{proof}
It will help to work with $\alpha_j$ and $\beta_j$ on the Fourier side. 
For $j\in \Z^+$, we let $\Lambda_j=\supp \widehat\alpha_j$ and $\Sigma_j=\supp \widehat\beta_j$.  As we already noted before, 
\[
\Lambda_j\subset [0,n_j-1], \quad \Sigma_j\subset [1,n_j]\,.
\] We will need more notation: given $N\in \Z$ and a set $S\subset \Z$, we write $N-S$ as a shorthand for a set $\{N-s, s\in S\}$.  From \eqref{sd1}, we get for $j\in \N$:
\begin{equation}\label{le}
\Lambda_j\subset \Lambda_{j-1}\cup (n_j-\Sigma_{j-1}), \quad \Sigma_j\subset \Sigma_{j-1}\cup (n_j-\Lambda_{j-1})\,
\end{equation}
and $\Lambda_0=\{0\}$ and $\Sigma_0\subset \{1\}$. Clearly, the sets in the above unions are disjoint due to the sparseness of $\{n_j\}$.

\begin{lemma}\label{suv} If $\alpha_j=:\sum_{0\le s}x_s^{(j)}z^s, \beta_j=:\sum_{0\le s}y_s^{(j)}z^s$, then for fixed $s$ we get $x_s^{(j)}=x_s$ and $y_s^{(j)}=y_s$ for large enough $j$. Moreover, $x:=\{x_s\}\in \ell^2(\Z^+)$ and $y:=\{y_s\}\in \ell^2(\Z^+)$.
\end{lemma}
\begin{proof}The first claim follows from   \eqref{sd1} and \eqref{le} which show that $\{x_s\}$ and $\{y_s\}$ are defined recursively. The second  claim follows  from the previous lemma.
\end{proof}

The following result was proved in \cite{Jelena}. It also comes as a corollary from our arguments above.

\begin{corollary} \label{c0}Given assumptions of Theorem \ref{t1}, the limit
 \begin{equation}\lim_{N\to\infty}\overbrace{F_{su(1,1)}^{\langle N\rangle}}\label{consa}
 \end{equation}
 exists for a.e. $z\in \T$.
\end{corollary}
\begin{proof}
 By the Carleson theorem on convergence of Fourier series \cite{carles}, the series
\[
\sum_{s=0}^\infty x_sz^s, \quad \sum_{s=0}^\infty y_sz^s
\]
converge for a.e. $z\in \T$. Since $\{\widetilde{\frak{a}}^{(*)}_n\}$ and $\{\widetilde{\frak{b}}_n\}$ are partial sums of these series, we get \eqref{consa}.
\end{proof}
\noindent {\bf Remark.} In \cite{Jelena}, it is proved that, provided that $q\ge 3$, the convergence of NLFT on a set of positive Lebesgue measure implies that the sparse NLFT data must be in $\ell^2(\Z)$.\medskip

 Let $\sigma$ be a probability measure on $\T$ with the infinite support (in the sense of cardinality). Denote the monic orthogonal polynomials by $\{\Phi_n(z,\sigma)\}$ and orthonormal polynomials by $\{\phi_n(z,\sigma)\}$, we use OPUC for both as shorthand. That is, 
\begin{eqnarray*}
\int_\T \Phi_n(z,\sigma)z^{-j}d\sigma=0, \quad \int_\T \phi_n(z,\sigma)z^{-j}d\sigma=0, \quad \forall j\in \{0,\ldots, n-1\},\quad\int_\T |\phi_n(z,\sigma)|^2d\sigma=1, 
\\ \deg (\Phi_n)=\deg(\phi_n)=n, \quad 
\text{coeff}_n(\Phi_n)=1, \quad \text{coeff}_n(\phi_n)>0\,,
\end{eqnarray*}
where $\text{coeff}_n(Q)$ denotes the $n$-th coefficient of the polynomial $Q$. For any polynomial $Q$, we let $Q^*:=z^n\overline{Q(\bar{z}^{-1})}$. Notice that such $\ast$-operation is different from the $(\ast)$-operation in \eqref{oper1} and it depends on $n\in \Z^+$. In fact, $Q^*=z^nQ^{(*)}$. Let $\D$ denote the open unit disc in $\C$.
The polynomials $\{\Phi_n(z,\sigma)\}$ satisfy a recursion (\cite{bs}, see (1.5.10) on p. 56)
\[
\Phi_{n+1}=z\Phi_n-\overline{\gamma_n}\Phi_n^*, \, \Phi_0=1, \, n\in \Z^+\,,
\]
where the coefficients $\{\gamma_n\}|_{n=0}^\infty\in \D^\infty$ are called recursion (aka Schur or Verblunsky) parameters. It is known that the set of measures $\sigma$ and the set of sequences $\{\gamma_n\}|_{n=0}^\infty\in \D^\infty$ are bijectively equivalent.
We get the following two results.

 \begin{corollary}\label{c1} Suppose $\{n_j\}$ and $\{\delta_j\}$ are taken as in the previous theorem. If the recurrence coefficients $\{\gamma_n(\sigma)\}$ of measure $\sigma$ satisfy $\gamma_{n_j}=\delta_j, \forall j\in \N$ and $\gamma_n=0$ for all other $n\in \Z^+$, then the limits $\lim_{n\to\infty}\phi_n^*(z,\sigma)$ and $\lim_{n\to\infty}\Phi_n^*(z,\sigma)$ exist for a.e. $z\in \T$.
 \end{corollary}
\begin{proof} The proof is immediate if one uses the connection between $SU(1,1)$ NLFT and the OPUC (\cite{den26}, Section 3 or \cite{tt}, Lecture 5).\end{proof}
 
 Next, we will discuss the $SU(2)$ NLFT. Take any sequence of complex numbers $G=\{G_n\}|_{n\in \Z}$. For any $N\in \Z^+$, consider $2\times 2$ matrices $Z_n(z,G^{\langle N\rangle}), n\in \Z$ defined by recursion
\begin{equation}\label{rec6}
 Z_n= \frac{1}{(1+|G_n^{\langle N\rangle}|^2)^{\frac 12}}
\left(\begin{array}{cc}
1& -\overline{G_n^{\langle N\rangle}}z^{-n}\\
G_n^{\langle N\rangle} z^n &1
\end{array}\right)Z_{n-1}, 
\end{equation}
with $Z_n=I:=\left(\begin{smallmatrix}1&0\\0&1
\end{smallmatrix}
\right)$ for $n< -N$. We have $Z_n\in SU(2)$ for $z\in \T$. The matrix $Z_n$ takes the form (we use the same notation as for the $SU(1,1)$ case but these are different quantities):
\[
Z_n(z,G^{\langle N\rangle})=:\left(\begin{array}{cc}
\frak{a}_n(z,G^{\langle N\rangle})& -\frak{b}^{(*)}_n(z,G^{\langle N\rangle})\\
\frak{b}_n(z,G^{\langle N\rangle}) &\frak{a}^{(*)}_n(z,G^{\langle N\rangle})
\end{array}\right)\,.
\]
Following \cite{AMT24,gev,amt}, we give a definition: \smallskip

\noindent {\bf Definition.} {\it For $G^{\langle N\rangle}$, define $SU(2)$ nonlinear Fourier transform (NLFT) on $\Z$ as the map 
\[
G^{\langle N\rangle}\mapsto \overbrace{G_{su(2)}^{\langle N\rangle}}:=\left(\begin{array}{c}
\frak{a}(z,G^{\langle N\rangle})\\
\frak{b}(z,G^{\langle N\rangle}) 
\end{array}\right)\,,
\]
where $\frak{a}(z,G^{\langle N\rangle}):=\frak{a}_\infty(z,G^{\langle N\rangle})$ and 
$\frak{b}(z,G^{\langle N\rangle})=\frak{b}_\infty(z,G^{\langle N\rangle})$.}
\smallskip
The proof of the following result repeats the proof of Corollary \ref{c0} word for word.
\begin{corollary}\label{t6}Suppose $\{n_j\}|_{j=0}^\infty$ is a subsequence of natural numbers that satisfies 
\begin{equation}\label{spar5}
\frac{n_{j+1}}{n_j}\ge q, \quad \forall j\in \Z^+, \, n_0=1\,,
\end{equation}
where $q\ge 2$. Take $\{\delta_n\}\in \ell^2(\Z^+)$ and let $G_{n_j}=\delta_j, \, \forall j\in \Z^+$ and set $G_n=0$ for all other $n\in \Z$. Then,
\[
 \,\,\overbrace{G_{su(2)}^{\langle N\rangle}}
\text{converges for a.e. $z\in \T$ when $N\to\infty$.}
\]
\end{corollary}\bigskip

Now, we are ready to prove Theorem \ref{t1}.

\noindent {\it Proof of Theorem \ref{t1}.}  
It  will be sufficient to establish the asymptotics of matrices $Y_n$. By interpolation, it is enough to consider $p=2\ell, \ell\in \N$ where the case $\ell=1$ follows from our arguments above. To explain the idea better, we start by taking $\ell=2$ which gives $p=4$. \smallskip
 
 Recall our definitions of $\alpha_j$ and $\beta_j$, and the coefficients $\{x_s\}$ and $\{y_s\}$ from Lemma \ref{suv}. We can partition the sequences $x$ and $y$ into blocks using the rapid growth of $n_j$ and write
 \[
 x:=\{x_s\}=\sum_{j\ge 0} P_j, \quad y:=\{y_s\}=\sum_{j\ge 0} Q_j\,,
 \]
where, see \eqref{sd1},  $\widecheck P_0=1, \widecheck Q_0=0$ and, for $j\in \N$, we have 
\begin{equation}\label{perh}
\widecheck P_j:=\delta_jz^{n_j}\overline{\beta}_{j-1}, \quad \widecheck Q_j:=\delta_jz^{n_j}\overline{\alpha}_{j-1}.
\end{equation} Moreover,  $\supp P_{j_1}\cap \supp P_{j_2}=\emptyset$ and $\supp Q_{j_1}\cap \supp Q_{j_2}=\emptyset$ if $j_1\neq j_2$. 
Hence, $\alpha_j=\widetilde{\frak{a}}^{(*)}_{n_j}=\sum_{s\le j} \widecheck{P}_s$ and $\beta_j=\widetilde{\frak{b}}_{n_j}=\sum_{s\le j} \widecheck{Q}_s$. We have
 $
 \|\widecheck P_j\|_{L^q(\T,m)}=|\delta_j|\|\beta_{j-1}\|_{L^q(\T,m)}, \, \|\widecheck Q_j\|_{L^q(\T,m)}=|\delta_j|\|\alpha_{j-1}\|_{L^q(\T,m)}
 $
 for every $q\ge 1$.  Take $\mu\in \N$ as a large parameter to be specified later. Define \[x^{(d)}:=\sum_{j\ge 0} P^{(d)}_j,\quad d\in \{0,\ldots,\mu-1\}, \quad P^{(d)}_j:=P_{j\mu+d},\]
 \[y^{(d)}:=\sum_{j\ge 0} Q^{(d)}_j,\quad d\in \{0,\ldots,\mu-1\}, \quad Q^{(d)}_j:=Q_{j\mu+d},\]
 and
 \[
\Delta_j:=\left(\sum_{|s-j\mu|\le 3\mu}|\delta_s|^2\right)^{\tfrac 12}.\]
Notice that $\sum_j\Delta_j^2\lesssim \mu \sum_{j}|\delta_j|^2$. 
We  write the sums
 \begin{equation}\label{pred}
f_j^{(d)}:= \sum_{s\le j}\widecheck P^{(d)}_s=f_{j-1}^{(d)}+\widecheck P^{(d)}_j, \quad g_j^{(d)}:=\sum_{s\le j}\widecheck Q^{(d)}_s=g_{j-1}^{(d)}+\widecheck Q^{(d)}_j
 \end{equation}
 so, for $d\in \{1,\ldots,\mu-1\}$,
 \begin{equation}\label{list1}
 \sum_{s\le \mu j+d-1}\widecheck Q_j=g_j^{(0)}+\ldots+g_j^{(d-1)}+g_{j-1}^{(d)}+\ldots+g_{j-1}^{(\mu-1)}
 \end{equation}
 and, for $d=0$,
 \begin{equation}\label{list2}
 \sum_{s\le \mu j-1}\widecheck Q_s=g_{j-1}^{(0)}+\ldots+g_{j-1}^{(\mu-1)}\,.
 \end{equation}
 Analogous formulas hold for the sums of $\widecheck P_j$. Notice that
 \begin{equation}\label{triangle}
 \sup_{j}\|\alpha_j\|_{L^q(\T,m)}\le \sum_{d=0}^{\mu-1}\sup_j\|f_j^{(d)}\|_{L^q(\T,m)}, \quad \sup_{j}\|\beta_j\|_{L^q(\T,m)}\le \sum_{d=0}^{\mu-1}\sup_j\|g_j^{(d)}\|_{L^q(\T,m)}
 \end{equation}
 by the triangle inequality. Our goal is to estimate $\|f_j^{(d)}\|_{L^4(\T,m)}$ and $\|f_j^{(d)}\|_{L^4(\T,m)}$.
To this end, we take the square of both sides in \eqref{pred}:
 \begin{eqnarray}\label{poy1}
 (f^{(d)}_j)^2=(f_{j-1}^{(d)})^2+(\widecheck{P}^{(d)}_j)^2+2f_{j-1}^{(d)}\widecheck{P}^{(d)}_j=:I_1+I_2+I_3\,,\\
 (g^{(d)}_j)^2=(g_{j-1}^{(d)})^2+(\widecheck{Q}^{(d)}_j)^2+2g_{j-1}^{(d)}\widecheck{Q}^{(d)}_j=:J_1+J_2+J_3\,.
 \end{eqnarray}
 Notice that $\supp \widehat I_1\subset [0,2n_{(j-1)\mu+d}]$ and $\supp \widehat I_2\subset [2(n_{j\mu+d}-n_{j\mu+d-1}),2n_{j\mu+d}]$ and $\supp \widehat I_3\subset [(n_{j\mu+d}-n_{j\mu+d-1},2n_{j\mu+d}] $.
 Now, we choose $\mu$ large enough to make sure that $I_1\perp I_2$ and $I_1\perp I_3$. This is possible due to \eqref{spar}. Similarly, we can guarantee that $J_1\perp J_2$ and $J_1\perp J_3$ by taking $\mu$ large.  Therefore, 
 \begin{equation}\label{cop1}
 \int |f^{(d)}_j|^4dm=\int |f_{j-1}^{(d)}|^4dm+\int |\widecheck{P}^{(d)}_j|^4dm +O\left(\int |f^{(d)}_{j-1}\widecheck{P}^{(d)}_j|^2dm\right)
 +O\left(\int |f^{(d)}_{j-1}|\cdot |\widecheck{P}^{(d)}_j|^3dm\right)\,.
 \end{equation}
 Now, we can use recursion \eqref{perh} to estimate the second, third, and the fourth terms as follows. If $d\in \{1,\ldots,\mu-1\}$, we use Young's inequality: $ab\le a^p/p+b^q/q, a,b\ge 0, p^{-1}+q^{-1}=1, p\in (1,\infty)$, along with \eqref{list1} to get
 \begin{eqnarray}
 \int |\widecheck{P}^{(d)}_j|^4dm\lesssim |\Delta_{j}|^4\int |g_j^{(0)}+\ldots+g^{(d-1)}_{j}+g_{j-1}^{(d)}+\ldots+g_{j-1}^{(\mu-1)}|^4dm, \\ \int |f^{(d)}_{j-1}\widecheck{P}^{(d)}_j|^2dm\lesssim \hspace{7cm}\\|\Delta_j|^2\left( \int |g_j^{(0)}+\ldots+g^{(d-1)}_{j}+g_{j-1}^{(d)}+\ldots+g_{j-1}^{(\mu-1)}|^4dm+ \int |f^{(d)}_{j-1}|^4dm\right),\\
 \int |f^{(d)}_{j-1}|\cdot |\widecheck{P}^{(d)}_j|^3dm\lesssim\hspace{7cm}  \\|\Delta_j|^3\left( \int | g_j^{(0)}+\ldots+g^{(d-1)}_{j}+g_{j-1}^{(d)}+\ldots+g_{j-1}^{(\mu-1)}  |^4dm+ \int |f^{(d)}_{j-1}|^4dm\right)
 \end{eqnarray}
 and, for $d=0$, we use \eqref{list2} to get
 \begin{eqnarray}
 \int |\widecheck{P}^{(0)}_{j}|^4dm\lesssim |\Delta_j|^4\int | g_{j-1}^{(0)}+\ldots+g^{(\mu-1)}_{j-1}  |^4dm, \\ \int |f^{(0)}_{j-1}\widecheck{P}^{(0)}_j|^2dm\lesssim |\Delta_j|^2\left( \int |  g_{j-1}^{(0)}+\ldots+g^{(\mu-1)}_{j-1}  |^4dm+ \int |f^{(0)}_{j-1}|^4dm\right),\\
  \int |f^{(0)}_{j-1}|\cdot |\widecheck{P}^{(0)}_j|^3dm\lesssim  |\Delta_j|^3\left( \int | g_{j-1}^{(0)}+\ldots+g^{(\mu-1)}_{j-1}   |^4dm+ \int |f^{(0)}_{j-1}|^4dm\right)\,.
 \end{eqnarray}
 Analogous bounds hold for $\{g^{(d)}_j\}$. Then, summing up in $d$ in \eqref{cop1}, we have
 \[
 \sum_{d=0}^{\mu-1}\int |f_j^{(d)}|^4dm\le  \left(\sum_{d=0}^{\mu-1}\int |f_{j-1}^{(d)}|^4dm\right)\left(1+C\Delta_j^2\right)+C_\mu\Delta_j^2\sum_{d=0}^{\mu-1}
  \int (|g^{(d)}_{j}|^4+|g^{(d)}_{j-1}|^4)dm
 \]
 and, similarly, one gets
 \[ \sum_{d=0}^{\mu-1}\int |g_j^{(d)}|^4dm\le  \left(\sum_{d=0}^{\mu-1}\int |g_{j-1}^{(d)}|^4dm\right)\left(1+C\Delta_j^2\right)+C_\mu\Delta_j^2\sum_{d=0}^{\mu-1}
  \int (|f^{(d)}_{j}|^4+|f^{(d)}_{j-1}|^4)dm\,.
\]
 We now add up these two bounds, use inequality $\sum_j \Delta_j^2<\infty$ and Lemma 4.1 from \cite{den26} to get
 \begin{equation}\label{cop2}
\sup_n \left(\sum_{d=0}^{\mu-1}\int |f_n^{(d)}|^4dm+\sum_{d=0}^{\mu-1}\int |g_n^{(d)}|^4dm\right)<\infty
 \end{equation}
 and, by \eqref{triangle}, $\sup_n \|\widetilde{\frak{a}}_n^{(*)}\|_{L^4(\T,m)}<\infty$ and $\sup_n \|\widetilde{\frak{b}}_n\|_{L^4(\T,m)}<\infty$. If we denote $\widetilde{\frak{a}}^{(*)}=\widecheck x$ and $\widetilde{\frak{b}}=\widecheck y$, the convergence $\|\widetilde{\frak{a}}_n^{(*)}-\widetilde{\frak{a}}^{(*)}\|_{L^4(\T,m)}\to 0$ and $\|\widetilde{\frak{b}}_n^{(*)}-\widetilde{\frak{b}}^{(*)}\|_{L^4(\T,m)}\to 0$ can be proved similarly using the Cauchy criterion and estimates on the tails of the series $\sum_j \widecheck P_j$ and $\sum_j \widecheck Q_j$. \smallskip

 For $\ell>2$, the argument is the same. Instead of \eqref{poy1}, we get
 \begin{equation}\label{poy2}
 (f^{(d)}_j)^{\ell}=(f_{j-1}^{(d)})^\ell+\ldots+(\widecheck{P}^{(d)}_j)^\ell\,,\quad  (g^{(d)}_j)^{\ell}=(g_{j-1}^{(d)})^\ell+\ldots+(\widecheck{P}^{(d)}_j)^\ell
 \end{equation}
 by Newton's binomial formula. Choosing  $\mu$ large enough, we can again  make sure that the first terms in both sums are orthogonal to all others. Then, taking  the square of the absolute values and using Young's inequality, one can obtain the required bounds just like in \eqref{cop1}-\eqref{cop2}. \qed

 \begin{theorem} Under the assumptions of  Corollary \ref{c1}, it holds that the orthogonality measure $\sigma$ satisfies $d\sigma=wdm$ and $w\in L^p(\T)$ with an arbitrary $p\in [2,\infty)$. Moreover, the sequences $\{\phi_n^*\}$ and $\{ \Phi_n^*\}$ converge in $L^p(\T,m)$.
 \end{theorem}

 \begin{proof}
 Assume without loss of generality that $\|\{\delta_j\}\|_{\ell^2}\le \frac 12$. Following \cite{khr}, we define the Wall polynomials associated with $\sigma$ by $A_j$ and $B_j$  and consider Bernstein-Szeg\H{o} approximations
 \[
 d\sigma_j=\sigma_j'dm=\frac{1}{|\phi_{j}|^2}dm, \quad r_j:=\frac{A_{j-1}}{B_{j-1}}\,.
 \]
 Then, by the formula (6.3) in \cite{khr}, we get
 \[
 \sigma_j'=\frac{1-|r_j|^2}{|1-zr_j|^2}\le \frac{1+|r_j|}{1-|r_j|}=\frac{(1+|r_j|)^2}{1-|r_j|^2}\stackrel{\text{(4.14) in \cite{khr}}}{\lesssim}|B_{j-1}|^2(1+|r_j|)^2\lesssim |B_{j-1}|^2\,.
 \]
 The connection between $\frak{a}_j,\frak{b}_j$ and $\phi_j^*,A_j,B_j$, already used in \cite{den26} in Section 3, shows the $L^p(\T,m)$ estimates hold for $A_j,B_j$ as well, i.e., $\sup_j\|A_j\|_{L^p(\T,m)}<\infty$ and $\sup_j\|B_j\|_{L^p(\T,m)}<\infty$ for all $p\in [2,\infty)$.
 Hence, $\sup_j\|\sigma_j'\|_{L^p(\T,m)}<\infty$ for every $p\in [1,\infty)$. Since $\sigma_j\stackrel{(*)}{\to}\sigma$ (see, e.g., Lemma 5.4 in \cite{khr}), we get our first statement by using the weak compactness of balls in $L^p(\T,m)$. The second claim follows from the previous theorem and the relation between the $SU(1,1)$ NLFT and the OPUC.
 \end{proof}
\bigskip\bigskip

\bibliographystyle{plain} 
\bibliography{bibfile}
\end{document}